\documentclass[lettersize, 10pt, conference]{IEEEtran}    
\IEEEoverridecommandlockouts                              %

\usepackage{graphicx,float,subfigure}
\usepackage{url}
\usepackage{amssymb,amsmath,amsfonts,layout}
\usepackage{amsthm} 
\usepackage{makeidx}
\usepackage{blkarray,multirow}
\usepackage{cite}
\usepackage{enumerate}
\usepackage{algpseudocode}
\usepackage{algorithm}
\usepackage{braket}
\usepackage{tikz}
\usepackage{lipsum}

\usepackage{ifthen,calc}
\usepackage{xstring}

\usetikzlibrary {positioning}
\usepackage{empheq}
\usepackage[most]{tcolorbox}

\newtcbox{\mybox}[1][]{%
    nobeforeafter, math upper, tcbox raise base,
    enhanced, colframe=blue!30!black,
    colback=blue!30, boxrule=1pt,
    #1}

\algnewcommand{\Inputs}[1]{%
	\State \textbf{Inputs:}
	\Statex \hspace*{\algorithmicindent}\parbox[t]{.8\linewidth}{\raggedright #1}
}
\algnewcommand{\Initialize}[1]{%
	\State \textbf{Initialize:}
	\Statex \hspace*{\algorithmicindent}\parbox[t]{.8\linewidth}{\raggedright #1}
}
\algnewcommand{\Outputs}[1]{%
	\State \textbf{Outputs:}
	\Statex \hspace*{\algorithmicindent}\parbox[t]{.8\linewidth}{\raggedright #1}
}

\usepackage{tikz}
\usetikzlibrary{shapes,arrows,positioning} 
\tikzset{
	>=stealth',
	block/.style={
		rectangle,
		rounded corners,
		draw=black, very thick,
		text width=12em,
		minimum height=3em,
		text centered},
	link/.style={
		->,
		thick,
		shorten <=2pt,
		shorten >=2pt},
	decision/.style={
		diamond,
		draw, very thick,
		fill=blue!20, 
		text width=8em,
		aspect=3,
		text centered}
}

\theoremstyle{plain}

\newtheorem{assumption}{Assumption}
\newtheorem{theorem}{Theorem}[section]

\newtheorem{lemma}[theorem]{Lemma}
\newtheorem{corollary}[theorem]{Corollary}

\newcommand{\longthmtitle}[1]{\mbox{}{\bf \textit{(#1).}}}

\newcommand{\oprocendsymbol}{\hbox{$\square$}}
\newcommand{\oprocend}{\relax\ifmmode\else\unskip\hfill\fi\oprocendsymbol}

\makeatletter
\newcommand{\pushright}[1]{\ifmeasuring@#1\else\omit\hfill$\displaystyle#1$\fi\ignorespaces}
\newcommand{\pushleft}[1]{\ifmeasuring@#1\else\omit$\displaystyle#1$\hfill\fi\ignorespaces}
\renewcommand*\env@matrix[1][*\c@MaxMatrixCols c]{%
	\hskip -\arraycolsep
	\let\@ifnextchar\new@ifnextchar
	\array{#1}}
\makeatother

\newcommand{\rank}{\operatorname{rank}}

\newcommand{\real}{\ensuremath{\mathbb{R}}}

\def\@fnsymbol#1{\ensuremath{\ifcase#1\or *\or \dagger\or \ddagger\or
   \mathsection\or \mathparagraph\or \|\or **\or \dagger\dagger
   \or \ddagger\ddagger \else\@ctrerr\fi}}
\newcommand{\ssymbol}[1]{^{\@fnsymbol{#1}}}

\newcommand\norm[1]{\left\lVert#1\right\rVert}

\renewcommand{\bar}{\overline}

\allowdisplaybreaks

\graphicspath{{epsfiles/}}

\begin{document}

\title{Robust Data-Driven Control with Noisy Data}
\author{Chin-Yao Chang and Andrey Bernstein
	\thanks{C.-Y, Chang, and A. Bernstein are with the National Renewable Energy Laboratory, Golden, CO 80401, USA (Emails: \{chinyao.chang,andrey.bernstein\}@nrel.gov).}
	\thanks{This work was authored in part by NREL, operated by Alliance for Sustainable Energy, LLC, for the U.S. Department of Energy (DOE) under Contract No. DE-AC36-08GO28308. Funding provided by DOE Office of Electricity, Advanced Grid Modeling Program, through agreement NO. 33652. The views expressed in the article do not necessarily represent the views of the DOE or the U.S. Government. The U.S. Government retains and the publisher, by accepting the article for publication, acknowledges that the U.S. Government retains a nonexclusive, paid-up, irrevocable, worldwide license to publish or reproduce the published form of this work, or allow others to do so, for U.S. Government purposes.}
}

\maketitle

\begin{abstract}
	This paper presents a robust data-driven controller design based on the noisy input-output data without assumptions on the statistical properties of the noises. We start with the direct data-representation of system models that take elements from behavioral system theory, followed by analyses of the upper bound of the ``modeling'' error with the data representation with presence of noises. Some pre-conditioning methods are put into the context based on how the derived bound is structured. We lastly leverage the upper bound to develop robust controllers that ride through the data noises. 
\end{abstract}

\section{Introduction}\label{sec.introduction}
In the last decade, there has been a major surge of interest in machine learning (ML) because it can find reasonable solutions to challenging and practical optimization problems which were otherwise unsolvable with standard optimization approaches. One can view ML as black-box identification of system models or optimal policies by using massive amounts of input-output data. Due to the lack of prior knowledge of targeted systems, artificial neural network (ANN) model structure has been widely adapted and received great success. However, there remain many applications (e.g. power systems) where partial system information is known in advance and the main challenges are instead on the robustness and resilience against all sorts of disturbances such as noisy data. We therefore put efforts on adding insights on robust controls with noisy data in this work.

The measured input-output data are usually used for system identification. One branch of system identification is the direct data representation of dynamical system models which were explored in the 80's~\cite{willems1979system,willems1986time_1,willems1986time_2,willems1987time_3}. The popularity of these behavioral system approaches have been recently revived. \cite{coulson2019data} applies behavioral system theory to develop data-enabled predictive control (DeePC), which is a data-driven  alternative to model predictive control (MPC). \cite{de2019formulas} draws connections between the behavioral system model and linear matrix inequalities (LMIs) stability analysis so that some classic controller designs can be applied to behavioral system models constructed from data. Series of follow-up works further expand these ideas to robust control~\cite{berberich2020data,berberich2020robust}, switched linear systems with unknown switching~\cite{dePersis2021lqr}, time-varying linear systems~\cite{verhoek2021fundamental} and real-time applications~\cite{rotulo2021online,baros2020online}. Fundamentally, behavioral system theory does not extract more information from data compared to classical system identification approaches. However, for controller design purposes, there are some subtle advantages of behavioral system theory-based approaches, detailed in~\cite{van2020data}. 

Noisy data draws errors on the identified system and in turns the performance of the controllers built upon the model are compromised. Based on the disturbance types, various specialized methods are available, e.g., principal component analysis (PCA)~\cite{vidal2005generalized}, regularization~\cite{poggio1987computational}, or methods for noisy labels~\cite{song2022learning}. In practice, it is very challenging to have prior knowledge on the noise or make proper assumptions on it. Both~\cite{berberich2020data} and~\cite{alanwar2022robust} are about robust data-driven predictive controls assuming no prior knowledge on the noise except an upper bound on the norm. \cite{berberich2020data} showed recursive feasibility and stability of the data-driven MPC formulated with the behavioral system model. \cite{alanwar2022robust}~proposed a zonotopic data-driven predictive control scheme that guarantees robust constraint satisfaction. Our focus in this paper is not on the predictive controls, but on the characterization of the modeling errors propagated from the noises, followed by a robust closed-loop feedback gain design.

\textit{Contributions}: Our first contribution is deriving an upper bound on data model representation errors that originate from disturbances of the collected data with no assumption on statistical properties of the noises. The upper bounds are conservative while we conjecture that there exists no much tighter bound from mathematical analysis because the condition number always shows up mechanically. Therefore, we next put some pre-conditioning methods in the context aiming to reduce the condition number. Although the main purpose of reducing the condition number is tightening the analytical upper bounds, numerical examples imply that minimizing the condition number of the data matrix could reduce actual modeling errors. We finally leverage the upper bounds for robust controller design for linear systems, with some additional comments on how those results can be extended to switched linear systems. 


\textit{Notations:} We denote by $\real$ and $\real_+$ the sets of real and positive real numbers, respectively. For a matrix $A$, we write $\|A\|$ and $\|A\|_F$, respectively, as its $2$-norm and  Frobenius norm. The pseudo inverse of $A$ is written as $A^\dagger$ with the subscript of $R$ or $L$ to indicate the right or left inverse if necessary. A matrix is called standardized if every column of it has unit $2$-norm. 

\section{Data Representation of System Modeling}\label{sec:data_rep}

In this section, we review the results in~\cite{de2019formulas} that connect Willems \textit{et al.}'s fundamental lemma for behavioral system theory and simple linear algebra, and then expand the discussion to switched linear systems.

\subsection{Data representation of LTI systems} \label{sec:LTI}
This subsection reviews some results in~\cite{de2019formulas} that will be used for the remainder of the paper. Consider the following simplified LTI system:
\begin{align}\label{eq:LTI0}
	x(k+1) = Ax(k) + Bu(k), 
\end{align}
where $x\in\real^n$, $u\in\real^m$, $A$ and $B$ are in proper dimensions. Here we assume that all the elements of $x$ can be measured. Define the collection of the measured data for the time horizon $T$ in the following:
\begin{subequations}\label{eq:data_matrices}
	\begin{align}
		&U_{0} = [u_{d}(0),\cdots,u_{d}(T-1)], \\ 
		&X_{0} = [x_{d}(0),\cdots,x_{d}(T-1)], \\  
		&X_{1} = [x_{d}(1),\cdots,x_{d}(T)],
	\end{align}
\end{subequations}
where $x_d(k)$ and $u_d(k)$ are, respectively, the state and control input data points at time $k$. If all the data points are noiseless, then by~\eqref{eq:LTI0} and the definition of the data matrices in~\eqref{eq:data_matrices}, we have
\begin{align}\label{eq:LTI0_data}
	X_{1} = \begin{bmatrix}  B & A \end{bmatrix}\begin{bmatrix} U_{0} \\ X_{0}   \end{bmatrix}.
\end{align}
Equation~\eqref{eq:LTI0_data} indicates that as long as $\begin{bmatrix} U_{0} \\ X_{0}  \end{bmatrix}$ is right invertible, we can find the system matrices, $A$ and $B$, straight from the data by
\begin{align}\label{eq:LTI0_data-2}
	\begin{bmatrix}  B & A \end{bmatrix} = X_{1}\begin{bmatrix} U_{0} \\ X_{0}   \end{bmatrix}^\dagger_R.
\end{align}
The full row rank (invertibility) of $\begin{bmatrix} U_{0} \\ X_{0}  \end{bmatrix}$ is formally stated in Assumption~\ref{as:datarank}.
\begin{assumption}\longthmtitle{Full row rank of the data matrix}\label{as:datarank}
	\begin{align}\label{as:rankeq}
		\text{rank}\left(\begin{bmatrix}U_{0} \\ X_{0}\end{bmatrix}\right) = m + n
	\end{align}
	\
\end{assumption}

As a side note, for the LTI system~\eqref{eq:LTI0} expanded with $y(k) = Cx(k) + Du(k)$ with only $y$ being measured instead of $x$, a similar rank condition to~\eqref{as:rankeq} posed on a \emph{Hankel matrix} constructed from $u_d$ can lead to a similar pure data representation of the system behavior in the sense of input-output pairs (assuming the linear system is observable and controllable). The full rank condition and the data representation of the system behavior are respectively stated as persistently exciting condition and  Willems \textit{et al.}'s fundamental lemma. More details are in~\cite{coulson2019data} and~\cite{de2019formulas}. 

\subsection{Switched linear systems with known modes}
\label{sec:sw_sys}
We consider the following switched linear system:
\begin{align}\label{eq:sw_sys}
	& x(k+1) = A_{\sigma(k)}x(k)  + B_{\sigma(k)}u(k), \\ \nonumber
	& \sigma(k) = f(x(k)),
\end{align}
where $x\in\real^n$, $u\in\real^m$, and $\sigma \in \Gamma:=[1,\cdots,\gamma]$ with some finite and known $\gamma$ number of modes. If one follows standard steps in Section \ref{sec:LTI}, identifying the system matrices $A_{\sigma}$ and $B_{\sigma}$ requires the system to stay in the same mode long enough to construct data matrices as shown in~\eqref{eq:data_matrices}. This restriction hinders the practicability of the data-driven control in the sense that the switching sequence might change frequently in practice. We  show next that if the system mode at each time $k$ is given, then even with arbitrary switching sequence, similar data matrices can still capture system matrices $A_{\sigma}$ and $B_{\sigma}$ for each individual mode.

Given the switch sequence $\sigma(0), \cdots, \sigma(T)$, we enumerate the input and measured data by $(u_{d}(0),\cdots,u_{d}(T-1))$ and $(x_{d}(0),\cdots,x_{d}(T))$, and then construct the following data matrices for each mode $i$:
\begin{align*}
	&U_{i,0} = [u_{i,d}(0),\cdots,u_{i,d}(T-1)], \\ 
	&X_{i,0} = [x_{i,d}(0),\cdots,x_{i,d}(T-1)],
\end{align*}
where for every $k = 0,\cdots,T$,
\begin{align*}
	u_{i,d}(k) &= \begin{cases} u_{d}(k) & \text{ if }\sigma(k)=i \\ 0 & \text{ otherwise}
	\end{cases}, \\ 
	x_{i,d}(k) &= \begin{cases} x_{d}(k) & \text{ if }\sigma(k)=i \\ 0 & \text{ otherwise}
	\end{cases}.
\end{align*}
The data matrices $U_{i,0}$ and $X_{i,0}$ are sparse by construction because column $k$ has all zero elements if $\sigma(k)\not=i$. We next make the following mild assumption on $U_{i,0}$ and $X_{i,0}$:
\begin{assumption}\label{as:datarank_sw}
	For each $i \in \Gamma$, 
	\begin{align}\label{as:rankeq_sw}
		\text{rank}\left(\begin{bmatrix}U_{i,0} \\ X_{i,0}\end{bmatrix} \right) = n + m.
	\end{align}
\end{assumption}
Assumption~\ref{as:datarank_sw} holds as long as $T$ is long enough such that there are sufficient occurrences for every mode $i\in\Gamma$. The $X_{1}$ matrix is still defined in the same way as the LTI case: $X_{1} = [x_{d}(1),\cdots,x_{d}(T)]$. The following theorem shows how the system matrices $[B_i \; A_i]$ for all $i\in\Gamma$ are identified.
\begin{lemma}\label{thm:id_sys}
	If Assumption~\ref{as:datarank_sw} holds, then for every $i\in\Gamma$, there exists $\begin{bmatrix} U_{i,0} \\ X_{i,0}    \end{bmatrix}^\dagger$ such that
	\begin{align}\label{eq:id_sys_i}
		[B_i \; A_i] = X_{1}\begin{bmatrix}
			U_{i,0} \\ X_{i,0}
		\end{bmatrix}^\dagger, \quad \forall i\in\Gamma.
	\end{align}
\end{lemma}
\begin{proof}
	There are infinite number of selections of the pseudo inverse of $\begin{bmatrix}U_{i,0} \\ X_{i,0} \end{bmatrix}$, while we choose the one with a structure such that~\eqref{eq:id_sys_i} follows. Define $U_{i,0}^s$ and $X^s_{i,0}$ respectively as the sub-matrices of $U_{i,0}$ and $X_{i,0}$ that collect all the non-zero columns of them. Given an arbitrary pseudo inverse of $\begin{bmatrix}
		U_{i,0}^s \\ X^s_{i,0}
	\end{bmatrix}$, $\begin{bmatrix}
		U_{i,0}^s \\ X^s_{i,0}
	\end{bmatrix}^\dagger$, we define $\begin{bmatrix} U_{i,0} \\ X_{i,0} \end{bmatrix}^\dagger$ such that its rows  associated with non-zero columns of $\begin{bmatrix} U_{i,0} \\ X_{i,0} \end{bmatrix}$ are comprised of $\begin{bmatrix}
		U_{i,0}^s \\ X^s_{i,0}
	\end{bmatrix}^\dagger$, and the remaining rows are all zero. This keeps the property that $\begin{bmatrix} U_{i,0} \\ X_{i,0} \end{bmatrix}\cdot \begin{bmatrix} U_{i,0} \\ X_{i,0} \end{bmatrix}^\dagger = I$, while more importantly,
	\begin{align}\label{eq:0crossing}
		\begin{bmatrix}
			U_{j,0} \\ X_{j,0}
		\end{bmatrix}
		\begin{bmatrix}
			U_{i,0} \\ X_{i,0}
		\end{bmatrix}^\dagger = 0, \quad\quad \text{ if }i\not=j.
	\end{align}
	Equation~\eqref{eq:0crossing} will be very useful in simplifying the following equality
	\begin{align}\label{eq:eq_sw_sys}
		X_{1} = \sum_{j\in\Gamma} [B_j \; A_j]\begin{bmatrix}
			U_{j,0} \\ X_{j,0}
		\end{bmatrix},
	\end{align}
	which is straight from~\eqref{eq:sw_sys}. Equation~\eqref{eq:id_sys_i} for each $i\in\Gamma$ follows by right multiplying~\eqref{eq:eq_sw_sys} by $\begin{bmatrix} U_{i,0} \\ X_{i,0} \end{bmatrix}^\dagger$. 
\end{proof}
Applying Lemma~\ref{thm:id_sys}, one can derive the modeling for each mode from data.

\section{Controller Design with Noisy Data}
\label{sec:noisy_data}
In this section, we first derive a bound of the errors of data representation of system models. The bound is proportional to the condition number of the data matrix, so we extend some discussions on reducing the condition number. The improved condition number tightens the bound and increases the practicability of robust controller designs that leverage the bound.

\subsection{Bounds for system identification errors}
\label{sec:bounds_ID}
In this section, we consider LTI system~\eqref{eq:LTI0} for simplicity. If Assumption~\ref{as:datarank} holds and the data are noisy, the linear model identified from~\eqref{eq:LTI0_data-2} is not exactly the original $\begin{bmatrix} B & A \end{bmatrix}$. With presence of the inaccuracy of the data,~\eqref{eq:LTI0_data-2} is rewritten with slight abuse of notations 
\begin{align}\label{eq:LTI0_data-2_pert}
	[B^e \; A^e] = X_{1}\begin{bmatrix}
		U_{0} \\ X_{0}
	\end{bmatrix}^\dagger_R, \quad \begin{bmatrix}
		B & A
	\end{bmatrix} = X_{1}^\star\begin{bmatrix}
		U_{0}^\star \\ X_{0}^\star
	\end{bmatrix}^\dagger_R,
\end{align}
where $[B^e \; A^e]$ is the estimated system matrix. The actual system matrix, $\begin{bmatrix} B & A \end{bmatrix}$, can not be derived straight from the data because the accurate system values ($X_{1}^\star$ and $\begin{bmatrix}
	U_{0}^\star \\ X_{0}^\star \end{bmatrix}_R$) are unknown. If the discrepancy between the actual input-output values and measured data is too large, the collected data provide very little insight for any control. The following assumption on the boundedness of the discrepancy is therefore justified in the sense of usefulness of the data.
\begin{assumption}\longthmtitle{Bounds on the noisy data} \label{ass:Bdd_noise}
	There exist $\delta_X$ and $\delta_{UX}$  
	\begin{align}\label{eq:ass2-1}
		X^\star_{1} = X_{1} + \delta_X, \quad\quad  \begin{bmatrix}
			U^\star_{0} \\  X^\star_{0}
		\end{bmatrix}  = \begin{bmatrix}
			U_{0} \\ X_{0}
		\end{bmatrix} + \delta_{UX}
	\end{align}
	and
	\begin{align}\label{eq:ass2-2}
		\frac{\norm{\delta_X} }{\norm{X_{1}^\star}} \leq r_{X,1}, \quad \frac{\norm{\delta_{UX}} }{\norm{\begin{bmatrix}
					U_{0}^\star \\  X_{0}^\star
		\end{bmatrix}}} \leq r_{UX,0} < 1.
	\end{align}
\end{assumption}
Assumption~\ref{ass:Bdd_noise} only poses bounds on the norm of total noise of the data, so there could be a few outlier data points that can be smoothed out by the rest of accurate data points. We post additional condition of $r_{UX,0} < 1$ for the proof of Theorem~\ref{thm:1}, which provides a bound for the estimation error of the system matrix, $\delta_{BA} := \begin{bmatrix} B & A \end{bmatrix} - [B^e \; A^e]$.

\begin{theorem}\longthmtitle{Bound of the estimation error of the system model} \label{thm:1}
	If Assumptions~\ref{as:datarank} and  \ref{ass:Bdd_noise} hold and, 
	\begin{subequations}\label{eq:thm1-1}
		\begin{align}
			&\rank{(\begin{bmatrix}
					U^\star_{0} \\ X^\star_{0}
				\end{bmatrix} )} = m+n, \\ 
			&\norm{\begin{bmatrix} 
					U^\star_{0} \\ X^\star_{0}
			\end{bmatrix}} < \infty, \quad\quad \norm{\begin{bmatrix}
					U_{0} \\ X_{0}
			\end{bmatrix}} < \infty,
		\end{align}
	\end{subequations}
	then
	\begin{align}\label{eq:thm1-2}
		\frac{\norm{\delta_{BA} } }{ \norm{ \begin{bmatrix}
					B & A \end{bmatrix} } }
		\leq  c_{UX}\frac{r_{X,1} + r_{UX,0}}{1 - r_{UX,0}},
	\end{align}
	where $c_{UX}= \norm{ \begin{bmatrix} U_{0} \\ X_{0}\end{bmatrix} } \norm{\begin{bmatrix} U_{0} \\ X_{0} \end{bmatrix}^\dagger_R}$ is the condition number.
\end{theorem}
\begin{proof}
	The key element of the proof is estimating how the errors of the data propagate to the pseudo inverse of the data matrix $\begin{bmatrix} U_{0} \\ X_{0}\end{bmatrix}$ when computing $\begin{bmatrix} B & A \end{bmatrix}$. Defining a matrix $\delta_{UX^\dagger}$ such that$\begin{bmatrix} U^\star_{0} \\ X^\star_{0} \end{bmatrix}^\dagger_R = 
	\begin{bmatrix} U_{0} \\ X_{0} \end{bmatrix}^\dagger_R + \delta_{UX^\dagger},$ we have
	\begin{align}\label{eq:diff_UX}
		&\begin{bmatrix} U_{0}^\star \\ X_{0}^\star\end{bmatrix}\begin{bmatrix} U_{0}^\star \\ X_{0}^\star\end{bmatrix}^\dagger_R = I \nonumber \\  \Rightarrow &\Big( \begin{bmatrix} U_{0} \\ X_{0} \end{bmatrix} + \delta_{UX} \Big) \Big( \begin{bmatrix} U_{0} \\ X_{0} \end{bmatrix}^\dagger_R + \delta_{UX^\dagger} \Big)  = I \nonumber \\ \nonumber
		\Rightarrow & \begin{bmatrix} U_{0} \\ X_{0}\end{bmatrix}\begin{bmatrix} U_{0} \\ X_{0}\end{bmatrix}^\dagger_R \hspace{-2mm} + \delta_{UX} \begin{bmatrix} U_{0} \\ X_{0}\end{bmatrix}^\dagger_R \hspace{-2mm} + \begin{bmatrix} U_{0} \\ X_{0}\end{bmatrix} \delta_{UX^\dagger} + \delta_{UX}\delta_{UX^\dagger} = I \nonumber \\ 
		\Rightarrow & \delta_{UX^\dagger} = -\begin{bmatrix} U_{0}^\star \\ X_{0}^\star\end{bmatrix}^\dagger_L\delta_{UX} \begin{bmatrix} U_{0} \\ X_{0}\end{bmatrix}^\dagger_R \nonumber \\
		\Rightarrow & \delta_{UX^\dagger} = -\begin{bmatrix} U_{0}^\star \\ X_{0}^\star\end{bmatrix}^\dagger_R\delta_{UX} \begin{bmatrix} U_{0} \\ X_{0}\end{bmatrix}^\dagger_R,
	\end{align}
	where from the third to the fourth line we use the property that $\begin{bmatrix} U_{0} \\ X_{0}\end{bmatrix}\begin{bmatrix} U_{0} \\ X_{0}\end{bmatrix}^\dagger_R = I$. By definitions of $\delta_{BA}$, $\begin{bmatrix}B & A
	\end{bmatrix}$, and $[B^e \; A^e]$, we also have
	\begin{align}\label{eq:diff_data_model}
		\delta_{BA} &= X^\star_{1}\begin{bmatrix}
			U_{0}^\star \\ X_{0}^\star
		\end{bmatrix}^\dagger_R - X_{1}  \begin{bmatrix}
			U_{0} \\ X_{0}
		\end{bmatrix}^\dagger_R  \nonumber \\ \nonumber
		&\hspace{-7mm}= (X_{1} + \delta_X) \Big( \begin{bmatrix}
			U_{0} \\ X_{0}
		\end{bmatrix}^\dagger_R + \delta_{UX^\dagger} \Big) - X_{1}  \begin{bmatrix}
			U_{0} \\ X_{0}
		\end{bmatrix}^\dagger_R  \\
		&\hspace{-7mm}=  \delta_X \begin{bmatrix}
			U_{0}^\star \\ X_{0}^\star
		\end{bmatrix}^\dagger_R + X_{1} \delta_{UX^\dagger}.
	\end{align}
	Substituting~\eqref{eq:diff_UX} to~\eqref{eq:diff_data_model} gives
	\begin{align}\label{eq:diff_data_model-2}
		\delta_{BA} &= \delta_X\begin{bmatrix}
			U_{0}^\star \\ X_{0}^\star
		\end{bmatrix}^\dagger_R \hspace{-3mm} - (X^\star_{1} -\delta_X)  \begin{bmatrix} U_{0}^\star \\ X_{0}^\star \end{bmatrix}^\dagger_R\delta_{UX} \begin{bmatrix} U_{0} \\ X_{0}\end{bmatrix}^\dagger_R  \nonumber \\ 
		&\hspace{-9mm}= \delta_X\begin{bmatrix}
			U_{0}^\star \\ X_{0}^\star
		\end{bmatrix}^\dagger_R \hspace{-3mm} - \begin{bmatrix}
			B & A \end{bmatrix}\delta_{UX} \begin{bmatrix} U_{0} \\ X_{0}\end{bmatrix}^\dagger_R + \delta_X \delta_{UX^\dagger} \nonumber \\
		&\hspace{-9mm}= \delta_X\begin{bmatrix}
			U_{0} \\ X_{0}
		\end{bmatrix}^\dagger_R \hspace{-3mm} - \begin{bmatrix}
			B & A \end{bmatrix}\delta_{UX} \begin{bmatrix} U_{0} \\ X_{0}\end{bmatrix}^\dagger_R.
	\end{align}
	Taking the norm on both sides of~\eqref{eq:diff_data_model-2} gives
	\begin{align}\label{eq:diff_data_model-3}
		& \norm{\delta_{BA}  } 
		\leq \norm{\delta_X} \norm{\begin{bmatrix}
				U_{0} \\ X_{0}
			\end{bmatrix}^\dagger_R} + \norm{ \begin{bmatrix}
				B & A \end{bmatrix} } \norm{ \delta_{UX} } \norm{\begin{bmatrix} U_{0} \\ X_{0}\end{bmatrix}^\dagger_R}  \\ \nonumber
		\Rightarrow & \frac{\norm{\delta_{BA} } }{ \norm{ \begin{bmatrix}
					B & A \end{bmatrix} } }
		\leq \frac{\norm{\delta_X} }{\norm{ \begin{bmatrix}
					B & A \end{bmatrix} }} \norm{\begin{bmatrix}
				U_{0} \\ X_{0}
			\end{bmatrix}^\dagger_R} + \norm{ \delta_{UX}  }  \norm{ \begin{bmatrix} U_{0} \\ X_{0}\end{bmatrix}^\dagger_R}   \nonumber 
	\end{align}
	Substituting $\norm{ \begin{bmatrix} U_{0}^\star \\ X_{0}^\star\end{bmatrix} }  \geq \frac{\|X^\star_{1}\|}{ \|\begin{bmatrix} B & A \end{bmatrix}\| }$ to~\eqref{eq:diff_data_model-3}, we get
	\begin{align}\label{eq:diff_data_model-4}
		\frac{\norm{\delta_{BA} } }{ \norm{ \begin{bmatrix}
					B & A \end{bmatrix} } }
		&\leq \norm{ \begin{bmatrix} U_{0}^\star \\ X_{0}^\star \end{bmatrix} } \norm{\begin{bmatrix}
				U_{0} \\ X_{0}
			\end{bmatrix}^\dagger_R} \Big( \frac{\norm{\delta_X} }{ \norm{X^\star_{1}} }   +  \frac{ \norm{ \delta_{UX}  } }{ \norm{\begin{bmatrix} U_{0}^\star \\ X_{0}^\star \end{bmatrix}} }  \Big)  \nonumber \\
		& = (r_{X,1} + r_{UX,0} ) \norm{\begin{bmatrix}
				U_{0} \\ X_{0}
			\end{bmatrix}^\dagger_R} \norm{ \begin{bmatrix} U_{0}^\star \\ X_{0}^\star \end{bmatrix} }.
	\end{align}
	It is more preferable to have the bound in $\norm{ \begin{bmatrix} U_{0} \\ X_{0} \end{bmatrix} }$ instead of the unknown $\norm{ \begin{bmatrix} U_{0}^\star \\ X_{0}^\star \end{bmatrix} }$. Therefore, we work around to bound $\norm{ \begin{bmatrix} U_{0}^\star \\ X_{0}^\star \end{bmatrix} }$ by $\norm{ \begin{bmatrix} U_{0} \\ X_{0} \end{bmatrix} }$,
	\begin{align}\label{eq:diff_data_model-5}
		&\norm{ \begin{bmatrix} U_{0}^\star \\ X_{0}^\star \end{bmatrix} } \leq \norm{ \begin{bmatrix} U_{0} \\ X_{0} \end{bmatrix} } + \norm{ \delta_{UX} } = \norm{ \begin{bmatrix} U_{0} \\ X_{0} \end{bmatrix} } + r_{UX,0} \norm{ \begin{bmatrix} U_{0}^\star \\ X_{0}^\star \end{bmatrix} }  \nonumber   \\
		&\Longrightarrow \norm{ \begin{bmatrix} U_{0}^\star \\ X_{0}^\star \end{bmatrix} } \leq \frac{1}{1 - r_{UX,0}} \norm{\begin{bmatrix} U_{0} \\ X_{0} \end{bmatrix} }.
	\end{align}
	Substituting~\eqref{eq:diff_data_model-5} to~\eqref{eq:diff_data_model-4} completes the proof.
\end{proof}

The caveat of the bound  Theorem~\ref{thm:1} provides,~\eqref{eq:thm1-2}, is that the denominator, $\begin{bmatrix}B & A \end{bmatrix}$, is unknown. By imposing a stronger assumption, Corollary~\ref{cor:1} gives a bound on $\delta_{BA}$ that is more practically implementable.
\begin{corollary}
	\longthmtitle{Bound of the estimation error of the system model} \label{cor:1}
	If Assumptions~\ref{as:datarank}, \ref{ass:Bdd_noise}, Eq.~\eqref{eq:thm1-1} hold, and 
	\begin{align}\label{eq:cor1-1}
		c_{UX}\frac{r_{X,1} + r_{UX,0}}{1 - r_{UX,0}} := \bar{c}< 1,
	\end{align}
	then
	\begin{align}\label{eq:cor1-2}
		\norm{\delta_{BA}
		} \leq \frac{\bar{c}}{1 - \bar{c}}\norm{\begin{bmatrix} B^e & A^e \end{bmatrix}
		}
	\end{align}
\end{corollary}
\begin{proof}
	The proof is straightforward.  By~\eqref{eq:thm1-2}, we have
	\begin{align*}
		&\norm{\delta_{BA}
		} \leq \bar{c}\norm{\begin{bmatrix} B & A \end{bmatrix} } \\ 
		\Longrightarrow &\norm{\delta_{BA}
		} \leq \bar{c}\Big(\norm{\begin{bmatrix} B^e & A^e \end{bmatrix} } + \norm{\delta_{BA} } \Big).
	\end{align*}
	Rearranging the terms in the equation above completes the proof.
\end{proof}
Note that for $\bar{c} < 1$, $r_{X,1}$ should be strictly less than one because the condition number $c_{UX} \geq 1$. Therefore, although we do not specifically impose an upper bound on $r_{X,1}$, it practically needs to be small enough so that the bound~\eqref{eq:cor1-2} can be derived. In general, there could be some other ways to bound the error. However, to the best of our knowledge, the tightness of the bounds is predominated by the condition number $c_{UX}$ and there is no way to remove the condition number from the bounds. Furthermore, the condition number $c_{UX}$ is determined by the raw data and there is not much control over it. Though there is some research on  \emph{effective} condition number for positive definite matrices that can potentially tighten the bounds~\cite{chan1988effectively,li2007effective}, how the concept extends to common rectangular matrices is unclear. We next discuss two routes to pre-conditioning the data matrix. 

\subsection{Pre-conditioning the data matrix}
There are two benefits of pre-conditioning the data matrix by reducing its condition number. One is a (numerically) tighter bound for robust controller design. The other is that the actual modeling error propagated from the noisy data can be reduced. The pre-conditioning is non-trivial, so we only go through some potentially useful methods that reduce the condition number. 

The first way to pre-condition the data matrix is by appropriate  selection of data points. Recall that the only requirement for the data representation of the original system model is the full row rank of $\begin{bmatrix} U_{0} \\ X_{0} \end{bmatrix}$ and we construct $U_0$ (or $X_0$) by using all the data from $k=0,\cdots,T$. The idea is that by selecting a subset of the $T+1$ data points to construct $U_0$ (or $X_0$), the condition number can become smaller, while the full row rank condition still holds. \cite{tropp2009column} provides some useful insights in this route, particularly the following theorem:
\begin{theorem}\longthmtitle{Bougain-Tzafriri \cite{tropp2009column}} 
	\label{thm:BT}
	Suppose matrix $A$ is standardized. Then there is a set $\tau$ of column indices
	for which
	\begin{align*}
		|\tau| \geq c\cdot  \frac{\|A\|_F}{\|A\|}
	\end{align*}
	such that the sub-matrix of $A$ indexed by $\tau$ has the condition number less than or equal to $\sqrt{3}$.
\end{theorem}
The constant $c$ in Theorem~\ref{thm:BT} refers to a positive, universal constant. The upper bound of $\sqrt{3}$ is very decent such that the bound in Theorem~\ref{thm:1} is  tight in the sense that there is limited amplification of data errors toward the modeling error. What makes Theorem~\ref{thm:1} more compelling is that there is an algorithmic version available~\cite[Algorithm~2]{tropp2009column}. However, practically, we may not find a feasible selection of the columns that gives the bound close to $\sqrt{3}$. One of the main reasons is that Theorem~\ref{thm:BT} (or~\cite[Algorithm~2]{tropp2009column}) counts the option of non-full row rank selection of columns, or vertical matrices such that Assumption~\ref{as:datarank} does not hold.  Nevertheless, one can modify~\cite[Algorithm~2]{tropp2009column} with an additional constraint on the number of columns in an attempt to improve the condition number by not using the full data set. 

Another way to improve the condition number is diagonal scaling. The goal is finding diagonal matrices, $D_L$ and $D_R$, such that the condition number of $\widehat{\begin{bmatrix} U_0 \\ X_0 \end{bmatrix}}:= D_L \begin{bmatrix} U_0 \\ X_0 \end{bmatrix} D_R$ is smaller than $\begin{bmatrix} U_0 \\ X_0 \end{bmatrix}$. This diagonal rescaling does not change the structure of the linear equality that we aim to solve:
\begin{align}
	&X_1 = \begin{bmatrix} B & A \end{bmatrix} \begin{bmatrix}
		U_0 \\ X_0
	\end{bmatrix} \nonumber \\
	\Longleftrightarrow& X_1 D_R = \begin{bmatrix} B & A \end{bmatrix} D_L^{-1} \Big(D_L\begin{bmatrix}
		U_0 \\ X_0 \end{bmatrix} D_R\Big) \nonumber \\
	\Longleftrightarrow& \widehat{X}_1 = \widehat{\begin{bmatrix} B & A \end{bmatrix} } \widehat{\begin{bmatrix}
			U_0 \\ X_0 \end{bmatrix} }, \label{eq:rescale_LTI}
\end{align}
where we define $\widehat{X}_1 = X_1 D_R$ and $\widehat{\begin{bmatrix} B & A \end{bmatrix}} = \begin{bmatrix} B & A \end{bmatrix} D_L^{-1}$. By repeating the steps in section~\ref{sec:bounds_ID} for~\eqref{eq:rescale_LTI} instead of~\eqref{eq:LTI0_data}, we can get the bound of the error term $\delta_{\widehat{BA}}$ relative to $\widehat{\begin{bmatrix} B & A \end{bmatrix}}$ as in~\eqref{eq:thm1-2} in Theorem~\ref{thm:1}, which is tighter than the original one in the sense that the condition number for $\widehat{\begin{bmatrix} U_0 \\ X_0 \end{bmatrix}}$ is smaller than $\begin{bmatrix} U_0 \\ X_0 \end{bmatrix}$. Note that we can not conclude analytically that the resulting $\begin{bmatrix} B & A \end{bmatrix}$ deduced from $\widehat{\begin{bmatrix} B & A \end{bmatrix}}$ has a smaller modeling error (originated from noisy data) compared to the case without the diagonal scaling. However, numerical studies imply that diagonal scaling could reduce the errors. More details about heuristic algorithms for the diagonal scaling can be found in~\cite{takapoui2016preconditioning,bradley2010algorithms}. 


\subsection{Robust controller design for LTI systems}
In this subsection, we leverage the modeling bounds in section~\ref{sec:bounds_ID} for a robust controller design for LTI systems. Our goal here is designing a feedback gain $K\in \real^{m\times n}$ such that~\eqref{eq:LTI0} is stable without knowing $A$ and $B$ matrices but some noisy input-output data that construct $U_0$, $X_0$, and $X_1$. We assume that there is an upper bound on the feedback gain, denoted as $\|K\| \leq \bar{K}$. Theorem~\ref{thm:robust_LTI} shows a formulation that finds the robust feedback gain, $K$.
\begin{theorem}\longthmtitle{Data-driven robust controller for LTIs}\label{thm:robust_LTI}
	If Assumption~\ref{as:datarank}, \ref{ass:Bdd_noise} holds and $\exists\lambda\in\real_+$, $Q\in \real^{T \times n}$, $\Delta X_1\in \real^{n \times T }$ such that 
	\begin{subequations}
		\label{eq:Lyap_stab-3}
		\begin{align}
			&X_0 Q = I, \label{eq:Lyap_stab-3_1} \\ 
			&\begin{bmatrix}
				\bar{K}I & U_0 Q \\ (U_0 Q)^\top & \bar{K}I
			\end{bmatrix}  \succeq 0, \label{eq:Lyap_stab-3_2} \\ 
			&\begin{bmatrix}
				(1 -\lambda \bar{\Delta X_1})I & X_{1}Q + (1 -\lambda )\Delta X_1 Q \\ 
				{Q}^\top { X_{1}  }^\top \hspace{-2mm}+ (1 \hspace{-1mm}-\lambda) {Q}^\top {\Delta X_1}^\top \hspace{-1mm}  & (1 -\lambda \bar{\Delta X_1})I
			\end{bmatrix}\succeq 0, \label{eq:Lyap_stab-3_3}
		\end{align}
	\end{subequations}
	where  $\bar{\Delta X_1} = \frac{\bar{c}}{1 - \bar{c}} \norm{\begin{bmatrix} B^e & A^e \end{bmatrix}}\norm{\begin{bmatrix} \bar{K} \\ I \end{bmatrix}}$. Then $u = Kx$, $K:= U_0Q(X_0Q)^{-1}$, stabilizes system~\eqref{eq:LTI0}.
\end{theorem}
\begin{proof}
	A stability condition for a feedback control $u = Kx$ is given by
	\begin{align}\label{eq:Lyap_stab}
		\exists P \succeq 0 \text{ s.t. } (A +BK)P(A +BK)^\top \preceq P.
	\end{align}
	Finding a $K$ that stabilizes the linear system is all about solving linear matrix inequalities shown in~\eqref{eq:Lyap_stab} when $A$ and $B$ are known. Similar to the method in~\cite{de2019formulas}, we define
	\begin{align}\label{eq:GainK}
		\begin{bmatrix}
			K \\ I 
		\end{bmatrix} = \begin{bmatrix}
			U_{0} \\ X_{0}
		\end{bmatrix} G,
	\end{align}
	which leads to
	\begin{align}\label{eq:data_rep}
		&A + B K = \begin{bmatrix} B & A \end{bmatrix} \begin{bmatrix} K \\ I \end{bmatrix} =\begin{bmatrix} B & A \end{bmatrix} \begin{bmatrix}
			U_{0} \\ X_{0} 
		\end{bmatrix} G   \\  \nonumber
		&=\Big( \begin{bmatrix} B^e & \hspace{-3mm} A^e \end{bmatrix} \begin{bmatrix}
			U_{0} \\ X_{0} 
		\end{bmatrix} + \delta_{BA} \begin{bmatrix}
			U_{0} \\ X_{0} 
		\end{bmatrix} \Big)G
	\end{align}
	By defining $Q = GP$ and $\Delta X_1= \delta_{BA} \begin{bmatrix}U_{0} \\ X_{0} \end{bmatrix}$, we rewrite~\eqref{eq:Lyap_stab} in the following
	\begin{align}\label{eq:Lyap_stab-2}
		&\hspace{-3mm}\exists P \succeq  0 \text{ s.t. } (X_1 \hspace{-1mm}+ \Delta X_1)GP(P)^{-1}PG^\top (X_1 \hspace{-1mm}+ \Delta X_1)^\top \preceq P  \nonumber \\
		\Leftrightarrow &\exists P \succeq 0 \text{ s.t. } (X_1 + \Delta X_1)Q(P)^{-1}Q^\top \hspace{-1mm} (X_1 + \Delta X_1)^\top \preceq P \nonumber \\
		\Leftrightarrow&\begin{bmatrix}
			X_{0} Q & (X_{1} + \Delta X_1) Q \\ 
			{Q}^\top \hspace{-1mm} { (X_{1} + \Delta X_1) }^\top & X_{0} Q
		\end{bmatrix}\succeq 0, \\  \nonumber
		& X_{0} Q = P \succeq 0,
	\end{align}
	where we apply Schur complement and use~\eqref{eq:GainK} in the derivations above. Note that the decision variable is changed from $K$ and $P$ to $Q$ as both $K$ and $P$ can be uniquely derived from $Q$. Specifically, $K = U_0G= U_0QP^{-1} = U_0Q(X_0Q)^{-1}$. The formulation~\eqref{eq:Lyap_stab-2} requires knowledge of $\Delta X_1$ which we do not have in general. An alternative and robust reformulation is to require~\eqref{eq:Lyap_stab-2} holds for all possible $\Delta X_1 Q$, where the results in Theorem~\ref{thm:1} or Corollary~\ref{cor:1} become handy. Without loss of generality, we let $P = I$, and then write the upper bound of $K$ and $\Delta X_1 Q$ as shown in the following:
	\begin{subequations}\label{eq:bdd_DX1}
		\begin{align}
			\norm{K} &= \norm{U_0 Q} \leq \bar{K} \\ 
			\norm{\Delta X_1 Q} &= \norm{ \delta_{BA} \begin{bmatrix}U_{0} \\ X_{0} \end{bmatrix} Q} \leq \norm{ \delta_{BA} }\norm{ \begin{bmatrix} K \\ I \end{bmatrix} } \leq \bar{\Delta X_1}.
		\end{align}
	\end{subequations}
	Rewriting~\eqref{eq:bdd_DX1} to LMIs gives~\eqref{eq:Lyap_stab-3_2} and
	\begin{align}
		& \begin{bmatrix}
			\bar{\Delta X_1}I & \Delta X_1 Q \\ (\Delta X_1 Q)^\top & \bar{\Delta X_1} I
		\end{bmatrix}  \succeq 0. \label{eq:bdd_DX1-2-2}
	\end{align}
	An alternative robust reformulation of~\eqref{eq:Lyap_stab-2} is to require that~\eqref{eq:Lyap_stab-2} holds for all $\Delta X_1$ satisfying~\eqref{eq:bdd_DX1-2-2} with $X_0 Q = P = I$. By applying S-procedure, we derive~\eqref{eq:Lyap_stab-3_3} and complete the proof.
\end{proof}

The matrix inequalities condition in Theorem~\ref{thm:robust_LTI} are nonlinear because of the bilinear term $(1 -\lambda )\Delta X_1 Q$, which pose some challenge to solve. An alternative convex formulation of~\eqref{eq:Lyap_stab-3} is available. Because $\Delta X_1$ does not show up in any other part of~\eqref{eq:Lyap_stab-3} other than~\eqref{eq:Lyap_stab-3_3}, we can define a new variable $\Delta Q:= \Delta X_1 Q$ to bypass the bilinear term $\Delta X_1 Q$ and retrieve $\Delta X_1$ afterward. The only remaining bilinear term is on $\lambda \Delta Q$. Numerically, we can  fix $\lambda$ at a small value and check the feasibility of~\eqref{eq:Lyap_stab-3}. If~\eqref{eq:Lyap_stab-3} is feasible for the given $\lambda$, we can still find a robust feedback control gain, $K$.

\subsection{Robust controller design for switched linear systems}

We extend the robust data-driven control to the switched linear system. The objective is to derive feedback controls $u = K_{i} x$ for $i\in\Gamma$ such that the switched linear system is stable under random switching. Finding a common Lyapunov function in the following guarantees the stability
\begin{align}\label{eq:Lyap_stab_sw}
	\exists P \succeq 0 \text{ s.t. } (A_i +B_i K_i)P(A_i +B_iK_i)^\top \hspace{-2mm}\preceq P, \;\forall i\in\Gamma.
\end{align}
Applying Lemma~\ref{thm:id_sys}, we can identify the system matrices for each mode $i\in\Gamma$ with data of randomly switching sequence. 
Similar to the last subsection, for each $i\in\Gamma$, we define $G_{i}$ by
\begin{align*}
	\begin{bmatrix}
		K_i \\ I 
	\end{bmatrix} = \begin{bmatrix}
		U_{i,0} \\ X_{i,0} 
	\end{bmatrix} G_i,
\end{align*}
which leads to 
\begin{align}\label{eq:data_rep_sw}
	&A_{i} + B_i K_i = [B_i \;\; A_i]\begin{bmatrix} K_i \\ I \end{bmatrix} = [B_i \;\; A_i] \begin{bmatrix} U_{i,0} \\ X_{i,0} \end{bmatrix} G_i \\  \nonumber
	&=  \Big([B_i^e \;\; A_i^e] + \delta_{BA_i} \Big) \begin{bmatrix} U_{i,0} \\ X_{i,0} \end{bmatrix} G_i \\  \nonumber
	&= \Bigg( \Big(X_{1} - \sum_{j\in\Gamma,j\not=i} [B_j^e \; A_j^e]\begin{bmatrix}
		U_{j,0} \\ X_{j,0}
	\end{bmatrix}\Big) + \delta_{BA_i}  \begin{bmatrix} U_{i,0} \\ X_{i,0} \end{bmatrix} \Bigg) G_i \\  \nonumber
	&= \Bigg( \Big(X_{1} - \hspace{-3mm} \sum_{j\in\Gamma,j\not=i} \hspace{-2mm} X_{1}\begin{bmatrix}
		U_{j,0} \\ X_{j,0}
	\end{bmatrix}^\dagger \begin{bmatrix}
		U_{j,0} \\ X_{j,0}
	\end{bmatrix} \Big) + \delta_{BA_i}  \begin{bmatrix} U_{i,0} \\ X_{i,0} \end{bmatrix} \Bigg) G_i \\  \nonumber
	&= \Bigg( X_{1} \Big( I- \hspace{-3mm} \sum_{j\in\Gamma,j\not=i} \begin{bmatrix}
		U_{j,0} \\ X_{j,0}
	\end{bmatrix}^\dagger \begin{bmatrix}
		U_{j,0} \\ X_{j,0}
	\end{bmatrix} \Big) + \delta_{BA_i} \begin{bmatrix} U_{i,0} \\ X_{i,0} \end{bmatrix} \Bigg)G_i.
\end{align}
Equation~\eqref{eq:data_rep_sw} is analogous to~\eqref{eq:data_rep} with the only difference on an additional term that multiplies with $X_1$ dependent on the mode $i\in\Gamma$. From this point on, we can introduce $Q_i$ for each $i\in\Gamma$ repeat the steps in the previous subsection to formulate a semidefinite programming similar to~\eqref{eq:Lyap_stab-3} for the robust control feedback gains. We skip those similar derivations that involve tedious notations with only marginal additional insights. 

\section{Numerical Results}
We consider a switched system in the form of~\eqref{eq:sw_sys} with $x\in\real^{20}$,  $u\in\real^{10}$ and $|\Gamma| = 5$. All the elements of $B_i$ for each $i\in\Gamma$ are randomly generated between $0$ and $0.1$;  $A_i$ for each $i\in\Gamma$ is generated as $A_i = 0.9 \cdot I + \Delta A_i$, where all the elements of $\Delta A_i$ are also randomly generated between $0$ and $0.1$. For the first 500 steps, we run the system under random control and switching for the purpose of collecting data. The measurement noise of the state $x$ is uniformly distributed between $-0.5\%$ and $0.5\%$ of the absolute value of $x$. We first apply the results in Section~\ref{sec:sw_sys} to identify $[B_i^e\;\; A_i^e]$ for each $i\in\Gamma$. We use the bound in Corollary~\ref{cor:1} for the purpose of robust controller design. The bound, however, is too conservative. Therefore, we applied the pre-conditioning methods including column selection using~\cite[Algorithm 2]{tropp2009column} (with an additional lower bound of the number of columns) and diagonal scaling using Ruiz algorithm in~\cite{takapoui2016preconditioning}. The column selection method was found not improving the condition number for a tighter bound. Our explanation for the ineffectiveness of the column selection are (i) the additional lower bound on the number of columns; (ii)~\cite[Algorithm 2]{tropp2009column} is randomized by nature and may only useful for certain classes of matrices. The diagonal scaling, on the other hand, improve the condition numbers of all the mode by a factor of around 10 as shown in Table~\ref{tab:1}. The improved condition numbers are directly reflected on tighter bounds of the estimation errors as shown in Table~\ref{tab:2}. In addition, the actual estimation errors are reduced marginally with the diagonal scaling, shown in Table~\ref{tab:3}. Those far better condition numbers benefit the robust controller design. We next apply the results in Section~\ref{sec:noisy_data} to find robust feedback control gains $K_i$ for all $i\in\Gamma$. Figure~\ref{fig:n20_m15_x} shows the trajectory of $x$ under the robust feedback control. As expected, $x$ converges to the origin within a moderate number of steps under the control. 
\begin{table}[H]
	\centering
	\begin{tabular}{|c|c|c|} \hline
		& w/o pre-conditioning &  w pre-conditioning \\ \hline 
		Mode 1 & 199.1373 &  21.0689 \\
		Mode 2 & 136.7279 &  16.3103 \\
		Mode 3 & 160.5263 &  18.2697 \\
		Mode 4 & 173.2082 &  18.6434 \\
		Mode 5 & 170.2047 &  20.3172 \\ \hline 
	\end{tabular} 
	\caption{The condition number of $[U_{i,0}^\top \; X_{i,0}^\top]^\top$ w/wo pre-conditioning the data matrices.}
	\label{tab:1}
\end{table}\vspace{-8mm}
\begin{table}[h]
	\centering
	\begin{tabular}{|c|c|c|} \hline
		& w/o pre-conditioning &  w pre-conditioning \\ \hline 
		Mode 1 & 4.0230 &  0.4256 \\
		Mode 2 & 2.7622 &  0.3295\\
		Mode 3 & 3.2430 &  0.3691\\
		Mode 4 & 3.4992 &  0.3766\\
		Mode 5 & 3.4385 &  0.4104 \\ \hline 
	\end{tabular} 
	\caption{The upper bounds of $\frac{\|\delta_{BA}\|}{\|[B \; A]\|}$ w/wo pre-conditioning the data matrices.}
	\label{tab:2}
\end{table}\vspace{-8mm}
\begin{table}[H]
	\centering
	\begin{tabular}{|c|c|c|c|} \hline
		& w/o pre-conditioning &  w pre-conditioning & $\%$ changes \\ \hline 
		Mode 1 & 0.0136  &  0.0115 & 15$\%$ \\
		Mode 2 & 0.0115  &  0.0095 & 17$\%$ \\
		Mode 3 & 0.0130  &  0.0125 & 4$\%$ \\
		Mode 4 & 0.0165  &  0.0155 & 6$\%$ \\
		Mode 5 & 0.0129  &  0.0125 & 3$\%$ \\ \hline 
	\end{tabular} 
	\caption{The value of $\frac{\|\delta_{BA}\|}{\|[B \; A]\|}$ w/wo pre-conditioning the data matrices.}
	\label{tab:3}
\end{table}\vspace{-8mm}

\begin{figure}[H]
	\centering
	\includegraphics[width=0.5\textwidth]{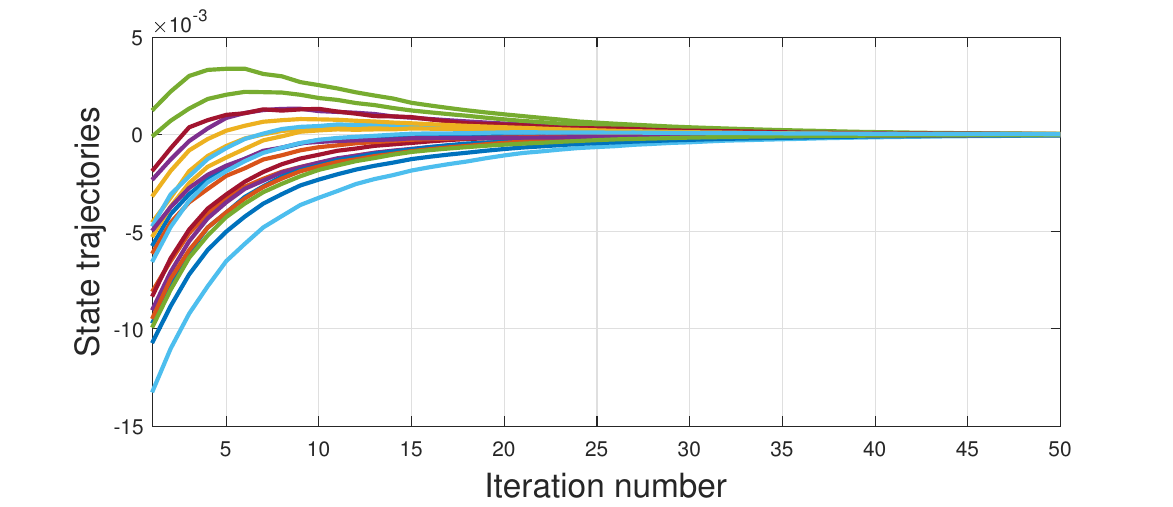}
	\caption{State trajectories under the proposed data-driven robust controller. }
	\label{fig:n20_m15_x}
\end{figure}

\section{Conclusion}
In this paper, we analyze system identification errors originated from noisy data and methods of pre-conditioning the data to improve the error bounds. The bounds on the inaccurate modeling are incorporated in robust controller design for LTI systems and switched linear systems. In the future, we will migrate the focus toward real-world applications and make necessary adjustments depending on the application needs.


\bibliographystyle{IEEEtran}
\bibliography{biblio.bib}
\end{document}